\documentclass[a4paper]{amsart}

\usepackage{color,tikz}

\DeclareMathOperator{\argmin}{arg min}

\newcommand{\R}{\mathbb{R}}

\newcommand{\wt}[1]{\widetilde{#1}}

\newcommand{\ov}[1]{\overline{#1}}

\newcommand{\mc}[1]{\mathcal{#1}}

\newcommand{\mb}[1]{\mathbb{#1}}

\newcommand{\eps}{\varepsilon}

\newcommand{\abs}[1]{\lvert#1\rvert}
\newcommand{\labs}[1]{\left\lvert\,#1\,\right\rvert}

\newcommand{\Lr}[1]{\left(#1\right)}

\newcommand{\diff}[2]{\dfrac{\partial #1}{\partial #2}}
\newcommand{\nm}[2]{\|#1\|_{#2}}

\newcommand{\set}[2]{\left\{\,#1\,\mid\,#2\,\right\}}
\def\al{\alpha}
\def\del{\delta}
\def\eps{\epsilon}
\def\na{\nabla}
\def\pa{\partial}
\def\lam{\lambda}
\def\Lam{\varLambda}

\def\x{\times}

\def\Om{\Omega}

\def\md{\mathrm{d}}
\def\dx{\,\mathrm{d}x}

\def\dsx{\,\md\sigma(x)}

\def\div{\operatorname{\na\cdot}}

\newcommand{\nn}{\nonumber}
\newcommand{\red}{\color{red}}

\theoremstyle{plain}
\newtheorem{theorem}{Theorem}[section]
\newtheorem{lemma}[theorem]{Lemma}

\theoremstyle{definition}

\theoremstyle{remark}

\begin{document}
\title[Deep Nitsche Method]{Deep Nitsche Method: Deep Ritz Method with Essential Boundary Conditions}
\author[Y. L. Liao]{Yulei Liao}
\address{LSEC, Institute of Computational Mathematics and Scientific/Engineering Computing, AMSS\\
 Chinese Academy of Sciences, No. 55, East Road Zhong-Guan-Cun, Beijing 100190, China\\
 and School of Mathematical Sciences, University of Chinese Academy of Sciences, Beijing 100049, China}
\email{liaoyulei@lsec.cc.ac.cn}

\author[P. B. Ming]{Pingbing Ming}
\address{LSEC, Institute of Computational Mathematics and Scientific/Engineering Computing, AMSS\\
 Chinese Academy of Sciences, No. 55, East Road Zhong-Guan-Cun, Beijing 100190, China\\
 and School of Mathematical Sciences, University of Chinese Academy of Sciences, Beijing 100049, China}
\email{mpb@lsec.cc.ac.cn}
%
\begin{abstract}
We propose a new method to deal with the essential boundary conditions encountered in the deep learning-based numerical solvers for partial differential equations. The trial functions representing by deep neural networks are non-interpolatory, which makes the enforcement of the essential boundary conditions a nontrivial matter. Our method resorts to Nitsche's variational formulation to deal with this difficulty, which is consistent, and does not require significant extra computational costs. We prove the error estimate in the energy norm and illustrate the method on several representative problems posed in at most $100$ dimension.
\end{abstract}

\subjclass{65N30, 65M12, 41A46, 35J25}
\keywords{Deep Nitsche Method, Deep Ritz Method, neural network approximation, mixed boundary conditions, curse of dimensionality.}
\thanks{This work was supported by the National Natural Science Foundation of China under the grant 11971467, and this work is also supported by Beijing Academy  of Artificial Intelligence (BAAI). The computations were done on the high performance computers of the State Key Laboratory of Scientific and Engineering Computing (LSEC), Chinese Academy of Sciences}
\maketitle
\section{Introduction}
Recently there has been a surge of interests in solving partial differential equations by deep learning-based numerical 
methods~\cite{EHJ:2017,Khoo:2017, EYu:2018, HanJE:2018, Siri:2018, Berg:2018, Fan:2019, Rassi:2019, Chen:20201, Chen:20202, Li:2020, Liu:2020, Wang:2020, Zhang:2020, ZBYZ:2020}, and we refer to~\cite{E:2020} for a review for this direction. These methods allow for the compositional construction of new approximation sets from various neural networks. Such constructions are usually free of a mesh so that they are in essence {\em meshless methods}~\cite{Babuska:2003}. The  trial functions in the approximation sets are in general non-interpolatory, which makes the implementation of the essential boundary conditions not an easy task. There are two main approaches to handle the essential boundary conditions in deep learning-based numerical methods{. O}ne is the conforming method, which exploits a supplementary neural network to make the functions in the trial set satisfy the boundary conditions {\em exactly}. This is the approach firstly proposed in~\cite{Lagaris:1998, Lagaris:2000} and recently further developed in~\cite{Khoo:2017, Berg:2018}. The conforming method usually involves an accurate evaluation of the distance function or a cut-off function, which is not easy for domain with complicated boundary geometry; See, e.g.,~\cite{Berg:2018}. Another one is the penalty method, which is a very general concept and belongs to the so-called nonconforming method~\cite{EYu:2018, Siri:2018, Rassi:2019, Zhang:2020, ZBYZ:2020}. An additional surface term is introduced into the variational formulation to enforce the boundary conditions. However, great care has to be taken to balance the different terms in the functional framework. Otherwise, this may cause problems for the existence and uniqueness of the solution~\cite{Akziz:1985, Bochev:2009}. Moreover, the penalty method usually leads to a sub-optimal rate of convergence as shown in~\cite{Babuska:1973} for finite element methods and as shown in~\cite{Babuska:2003} for the generalized finite element methods and meshless methods.

Compared to the penalty method, the Lagrange multiplier method treats the essential boundary conditions as a constraint in the minimization. This technique has been used to deal with the essential boundary conditions in finite element method~\cite{Babuska:1973b} and wavelet method~\cite{Dahmen:2001}. The optimal rate of convergence may be achieved if the approximation function spaces are chosen properly, which relies on the so-called inf-sup condition~\cite{Babuska:1973b, Dahmen:2001}. The Lagrange multiplier method may also be used to enforce boundary conditions in the neural-network based method provided that the resulting constrained minimization problem can be efficiently solved.

An efficient method for imposing the essential boundary conditions has been proposed by Nitsche in the early 1970's~\cite{Nitsche:1971} in the finite element method. It was quite unknown for many years, and was revived in~\cite{Stenberg:1995} by {\sc Stenberg}. He revealed the interesting relation between Nitsche's method and certain stabilized Lagrangian multiplier methods. More recent efforts on Nitsche's method have been devoted to deal with the elliptic interface problems and the unfitted mesh problems; we refer to~\cite{Burman:2012} for a review of the progress in this direction. In the context of the meshless method, Nitsche's idea has been proved to be an efficient approach to deal with the essential boundary conditions in the framework of a particle partition of unity method~\cite{Griebel:2003} as well as the generalized finite element method~\cite{Melenk:2005}.

In this work, we incorporate the idea of Nitsche into the framework of Deep Ritz Method~\cite{EYu:2018} to deal with the essential boundary conditions. This new algorithm is called Deep Nitsche Method. It also imposes the boundary conditions in a nonconforming way as the penalty method. In contrast to the penalty method, this method is consistent if the exact solution is smooth enough. 
The method is based on the energy formulation of Nitsche~\cite{Nitsche:1971}, which does not involve a Lagrange multiplier. Hence we need not  solve a constrained minimization problem, and the stochastic gradient descent (SGD) method may be used to solve the resulting minimization problem. To analyze the method, we exploit Nitsche's energy formulation instead of the Euler-Lagrange equations associated with the minimization problem, which in general does not exist for the deep Nitsche method because the trial function set formed by the deep neural network is a manifold instead of a space. We prove the energy error bound of the deep Nitsche method without taking into account the error caused by SGD. The error bound consists of two parts. The first part is the approximation error caused by the underlying deep neural network, and the second part is the estimation error, which comes from the numerically evaluation of the energy functional, equivalently, the loss function. Such error structure bears certain similarity with the {\em first Lemma of Strang}~\cite{Berger:1972}. We test the method with some mixed boundary value problems in two dimension with smooth solution and singular solution. The variational formulation may be adapted for solving nonlinear problem such as p-Laplace equation. We also apply the method to solve high dimensional problems with Dirichlet boundary condition. In all these cases, the solutions can be well approximated by the proposed method at a relative accuracy of $10^{-2}\sim 10^{-3}$, with only $10^3\sim 10^4$ parameters for 2d problems and $10^4\sim 10^5$ parameters for high dimensional problems. 

The rest of the paper is as follows. In the next part, we introduce the energy formulation of the method and an abstract error bound is proved by the aid of the energy formulation. In \S~\ref{sec:dnm}, we propose the deep Nitsche method, and then we present the numerical results in \S~\ref{sec:numerics} for solving some mixed boundary value problems with regular and singular solutions in two-dimension and also for problems in high-dimension up to $100$. In the last section, we conclude with some remarks.
\section{Nitsche's Variational Formulation}
We consider the following mixed boundary value problem
\begin{equation}\label{eq:mixbvp}
\left\{\begin{aligned}
-\div\Lr{A(x)\na u}&=f\quad&&\text{in\quad}\Omega,\\
u&=g_D\quad&&\text{on\quad}\Gamma_D,\\
\diff{u}{\nu}&=g_N\quad&&\text{on\quad}\Gamma_N,
\end{aligned}\right.
\end{equation}
where $\Omega$ is a bounded domain in $\mb{R}^d$, and $\Gamma_D\cup\Gamma_N=\pa\Omega$ and $\ov{\Gamma}_D\cap\ov{\Gamma}_N\not=\emptyset$. The conormal derivative of $u$ is defined as $\pa_{\nu}u=n_iA_{ij}\pa_{x_j}u$ with
$n=(n_1,\cdots,n_d)$ the outer normal of $\Om$. We assume that $A$ is a symmetric matrix with
\[
\lam\abs{\xi}^2\le A_{ij}(x)\xi_i\xi_j\le\Lam\abs{\xi}^2\qquad\text{a.e.\quad}x\in\Om\quad\text{and}\quad\xi\in\mb{R}^d.
\]

The minimization problem is defined as
\begin{equation}\label{eq:mini}
I[u_n]=\min_{v\in\mc{H}_n}I[v]
\end{equation}
with
\begin{align*}
I[v]&=\dfrac12\int_{\Om}A\na v\cdot\na v\dx+\dfrac{\beta}{2}\int_{\Gamma_D}(g_D-v)^2\dsx+\int_{\Gamma_D}(g_D-v)\pa_{\nu}v\dsx\\
&\quad-\Lr{\int_{\Om}fv\dx+\dfrac{\beta}{2}\int_{\Gamma_D}g_D^2\dsx+\int_{\Gamma_N}g_N v\dsx},
\end{align*}
where $\beta$ is a positive parameter. Here $\mc{H}_n$ is the set with neural network functions with $n$ the size of the set. For example, we define a set for a two-layer shallow network as
\[
\mc{H}_n{:}=\set{v=\sum_{i=1}^na_i\sigma(b_i\cdot x+c_i)}{a_i,c_i\in\mb{R},b_i\in\mb{R}^d}
\]
with the activation function $\sigma$. We assume that the activation function is smooth such that $\mc{H}_n\subset H^2(\Om)$. Therefore, $I[v]$ is well-defined for all $v\in\mc{H}_n$.

To study the Nitsche's variational problem, one usually resorts to the associated Euler-Lagrange equation as in~\cite{Melenk:2005}. Unfortunately, there is no such Euler-Lagrange equation for the minimization problem~\eqref{eq:mini} because $\mc{H}_n$ is a manifold instead of a subspace. This is one of the main difficulties in analyzing neural network-based numerical method. We overcome this difficult by exploiting the original energy formulation of Nitsche~\cite{Nitsche:1971}.
\begin{lemma}
The minimization problem~\eqref{eq:mini} is equivalent to
\begin{equation}\label{eq:originmini}
\wt{I}[u-u_n]=\min_{v\in\mc{H}_n}\wt{I}[u-v],
\end{equation}
where
\begin{equation}\label{eq:energy}
\wt{I}[v]{:}=\dfrac12\int_{\Omega}A\na v\cdot\na v\dx-\int_{\Gamma_D}v\diff{v}{\nu}\md\sigma(x)+\dfrac{\beta}{2}\int_{\Gamma_D}v^2\md\sigma(x),
\end{equation}
\end{lemma}

\begin{proof}
We start with
\begin{align*}
\wt{I}[u-v]&=\wt{I}[u]+\wt{I}[v]\\
&\quad-\int_{\Om}A\na u\cdot\na v\dx+\int_{\Gamma_D}\Lr{u\pa_{\nu}v+v\pa_{\nu}u}\md\sigma(x)-\beta\int_{\Gamma_D}uv\,\md\sigma(x).
\end{align*}
Multiplying~\eqref{eq:mixbvp}$_1$ by $v$, an integration by parts yields
\begin{align*}
\int_{\Om}A(x)\na u\cdot\na v\dx&=\int_{\Om}fv\dx+\int_{\pa\Om}\pa_{\nu}u v\md\sigma(x)\\
&=\int_{\Om}fv\dx+\int_{\Gamma_N}g_N v\md\sigma(x)+\int_{\Gamma_D}\pa_{\nu}uv\md\sigma(x).
\end{align*}
A combination of the above two equation yields
\begin{equation}\label{eq:iden}
\wt{I}[u-v]=\wt{I}[u]+I[v]
\end{equation}
with $I[v]$ given by~\eqref{eq:mini}. This proves the equivalence between the minimization problems~\eqref{eq:mini} and~\eqref{eq:originmini}.
\end{proof}

The following lemma states that Nitsche's method is consistent if the solution is smooth enough.
\begin{lemma}\label{lema:cons}
Let $\Omega$ be a Lipschitz domain. If $u\in H^2(\Omega)$, then $u$ is a critical point of the minimization problem
\begin{equation}\label{eq:consis}
\min_{v\in H^2(\Om)}\wt{I}[u-v].
\end{equation}
\end{lemma}

The condition in Lemma~\ref{lema:cons} may be replaced by a weaker condition: $u,v\in H^s(\Om)$ with $s>3/2$.

\begin{proof}
By the trace theorem~\cite{Adams:2003}, the assumption $u\in H^2(\Om)$ implies that the conormal derivative $\pa_{\nu}u$ is well-defined and $\pa_{\nu}u\in L^2(\Om)$. Therefore, it remains to prove that $w=0$ is a critical point of the minimization problem
\[
\min_{w\in H^2(\Om)}\wt{I}[w].
\]
The Euler-Lagrange equation associates with the minimization problem $\min_{w\in H^2(\Om)}\wt{I}[w]$ reads as
\[
\int_{\Om}A\na w\cdot\na v\dx+\int_{\Gamma_D}w(\beta v-\pa_{\nu}v)\dsx-\int_{\Gamma_D}\pa_{\nu}wv\dsx=0\qquad\text{for all\quad}v\in H^2(\Om).
\]
Note that $w,v\in H^2(\Om)$ guarantees that the Gauss-Green theorem holds, an integration by parts implies that $w$ satisfies~\eqref{eq:mixbvp} with $f=g_D=g_N=0$. Therefore, we conclude that $w\equiv 0$ by the uniqueness of the solution of the boundary value problem~\eqref{eq:mixbvp}.
\end{proof}

In what follows, we assume that the following {\em inverse trace inequality} is valid: There exists a constant $\gamma$ such that
\begin{equation}\label{eq:inverse}
\nm{\na v}{L^2(\pa\Om)}\le\gamma\nm{\na v}{L^2(\Om)}\qquad\text{for all\quad}v\in\mc{H}_n.
\end{equation}

We also make the following approximation assumptions:
\begin{equation}\label{eq:appass}
\begin{aligned}
\inf_{v\in\mc{H}_n}\Lr{\nm{\na(u-v)}{L^2(\pa\Om)}+\gamma\nm{\na(u-v)}{L^2(\Om)}}&\le\delta,\\
\inf_{v\in\mc{H}_n}\nm{u-v}{L^2(\pa\Om)}&\le\delta_1.
\end{aligned}
\end{equation}

We shall derive the error estimate by assuming the existence of the global minimizer $u_n^\ast$, which in general need not exits. However, for any $\eps>0$, there always exists an $\epsilon-$suboptimal global minimizer $u^\eps_n\in\mc{H}_n$ in the sense that $I[u_n^\eps]\le\inf_{v\in\mc{H}_n}I[v]+\eps$. We refer to~\cite{Houska:2019} for a detailed discussion on the properties of the  $\epsilon-$suboptimal global minimizer. All the error estimates remain valid if we replace the global minimizer to the $\epsilon-$suboptimal global minimizer. In what follows, without loss of generality, we assume the existence of at least one global minimizer for the minimization problem~\eqref{eq:mini}. Given the existence of the minimizer $u_n$, we exploit the minimization problem~\eqref{eq:mini} and the identity~\eqref{eq:iden} to obtain the error estimate in Theorem~\ref{thm:main1}. 
%
\begin{theorem}\label{thm:main1}
If \(
\beta>8\Lam^2\gamma^2/\lam,
\)
the inverse trace inequality~\eqref{eq:inverse} and the approximation assumption~\eqref{eq:appass} are valid, then the solution $u_n$ satisfies
\begin{equation}\label{eq:error1}
\nm{\na(u-u_n)}{L^2(\Om)}+\sqrt{\beta}\nm{u-u_n}{L^2(\Gamma_D)}
\le  C\Lr{\delta/\gamma+\delta/\sqrt\beta+\sqrt{\beta}\delta_1},
\end{equation}
where $C$ only depends on $\Lam$ and $\lam$.
\end{theorem}

The above error estimate~\eqref{eq:error1} may be written into a more convenient form:
\begin{align}\label{eq:error1a}
&\quad\nm{\na(u-u_n)}{L^2(\Om)}+\sqrt{\beta}\nm{u-u_n}{L^2(\Gamma_D)}\\
&\le C\inf_{v\in\mc{H}_n}
\Bigl(\Lr{\dfrac{1}{\gamma}+\dfrac{1}{\sqrt{\beta}}}
\Lr{\nm{\na(u-v)}{L^2(\Gamma_D)}+\gamma\nm{\na(u-v)}{L^2(\Om)}}+\sqrt{\beta}\nm{u-v}{L^2(\Gamma_D)}\Bigr).\nn
\end{align}

The error estimate is based on the equivalence between the minimization problems~\eqref{eq:mini} and~\eqref{eq:originmini}.

\begin{proof}
Denote $e{:}=u-u_n$. For any $v\in\mc{H}_n$, using the inverse trace inequality~\eqref{eq:inverse} and the approximation assumption~\eqref{eq:appass}$_1$, we obtain
\begin{align*}
\labs{\int_{\Gamma_D}e\pa_{\nu}e\md\sigma(x)}&\le\labs{\int_{\Gamma_D}e\pa_{\nu}(u-v)\md\sigma(x)}
+\labs{\int_{\Gamma_D}e\pa_{\nu}(v-u_n)\md\sigma(x)}\\
&\le\Lam\nm{e}{L^2(\Gamma_D)}\Lr{\nm{\na(u-v)}{L^2(\Gamma_D)}+\gamma\nm{\na(v-u_n)}{L^2(\Om)}}\\
&\le\Lam\gamma\nm{e}{L^2(\Gamma_D)}\nm{\na e}{L^2(\Om)}\\
&\quad+\Lam\nm{e}{L^2(\Gamma_D)}\Lr{\nm{\na(u-v)}{L^2(\Gamma_D)}+\gamma\nm{\na(u-v)}{L^2(\Om)}}\\
&\le\Lam\gamma\nm{e}{L^2(\Gamma_D)}\nm{\na e}{L^2(\Om)}
+\Lam\delta\nm{e}{L^2(\Gamma_D)}.
\end{align*}
Using~\eqref{eq:energy} and the above inequality, we obtain
\[
2\wt{I}[e]\ge\lam\nm{\na e}{L^2(\Om)}^2+\beta\nm{e}{L^2(\Gamma_D)}^2-2\Lam\gamma\nm{e}{L^2(\Gamma_D)}\nm{\na e}{L^2(\Om)}-2\Lam\delta\nm{e}{L^2(\Gamma_D)}.
\]
By Cauchy-Schwartz inequality and the fact that $\beta>8\Lam^2\gamma^2/\lam$, we obtain
\begin{align*}
2\wt{I}[e]&\ge\lam\nm{\na e}{L^2(\Om)}^2+\beta\nm{e}{L^2(\Gamma_D)}^2-\dfrac{\lam}{2}\nm{\na e}{L^2(\Om)}^2
-\dfrac{2\Lam^2\gamma^2}{\lam}\nm{e}{L^2(\Gamma_D)}^2\\
&\quad-\dfrac{\beta}{2}\nm{e}{L^2(\Gamma_D)}^2-2\Lam^2\del^2/\beta\\
&=\dfrac{\lam}2\nm{\na e}{L^2(\Om)}^2+\Lr{\dfrac{\beta}{2}-\dfrac{2\Lam^2\gamma^2}{\lam}}
\nm{e}{L^2(\Gamma_D)}^2-2\Lam^2\del^2/\beta\\
&\ge\dfrac{\lam}2\nm{\na e}{L^2(\Om)}^2+\dfrac{\beta}{4}\nm{e}{L^2(\Gamma_D)}^2-2\Lam^2\del^2/\beta.
\end{align*}

Next, using~\eqref{eq:originmini}, we obtain that for any $v\in\mc{H}_n$, there holds
\begin{align*}
2\wt{I}[e]\le 2\wt{I}[u-v]&\le\Lam\nm{\na(u-v)}{L^2(\Om)}^2+2\Lam\nm{u-v}{L^2(\Gamma_D)}\nm{\na(u-v)}{L^2(\Gamma_D)}\\
&\quad+\beta\nm{u-v}{L^2(\Gamma_D)}^2\\
&\le\dfrac{\Lam\delta^2}{\gamma^2}+2\Lam\delta_1\delta+\beta\delta_1^2,
\end{align*}
where we have used the approximation properties~\eqref{eq:appass} in the last step.

A combination of the above two inequalities gives
\begin{align*}
\dfrac{\lam}2\nm{\na e}{L^2(\Om)}^2+\dfrac{\beta}{4}\nm{e}{L^2(\Gamma_D)}^2
&\le\dfrac{\Lam\delta^2}{\gamma^2}+2\Lam\delta_1\delta+\beta\delta_1^2+2\Lam^2\del^2/\beta\\
&\le\dfrac{\Lam\delta^2}{\gamma^2}+3\Lam^2\delta^2/\beta+2\beta\delta_1^2.
\end{align*}
This implies
\[
\nm{\na(u-u_n)}{L^2(\Om)}\le\max(\sqrt{2\Lam/\lam},\sqrt{6\Lam^2/\lam},2/\sqrt{\lam})
\Lr{\delta/\gamma+\delta/\sqrt\beta+\sqrt{\beta}\delta_1},
\]
and
\[
\sqrt{\beta}\nm{u-u_n}{L^2(\Gamma_D)}\le 2\max(\sqrt{\Lam}, \sqrt{3}\Lam,\sqrt2)\Lr{\delta/\gamma+\delta/\sqrt\beta+\sqrt\beta\delta_1}.
\]

Combining the above two inequalities, we obtain~\eqref{eq:error1}.
\end{proof}

In what follows, we consider the numerical integration of Nitsche's formulation.
\begin{equation}\label{eq:nitscheMC}
I^\ast[u_n^\ast]=\min_{v\in\mc{H}_n}I^\ast[v],
\end{equation}
where $I^\ast[v]$ is the Monte Carlo approximation or Quasi-Monte Carlo approximation~\cite{Dick:2013} of $I[v]$ or other numerical integration schemes acting on $I[v]$.
\begin{theorem}\label{thm:main}
If the inverse trace inequality~\eqref{eq:inverse} is true, then there exists $C$ that depends only on $\Lam$ and $\lam$ such that the minimizer $u_n^\ast\in\mc{H}_n$ satisfies
\begin{equation}\label{eq:error3}
\begin{aligned}
&\quad\nm{\na(u-u_n^\ast)}{L^2(\Om)}+\sqrt{\beta}\nm{u-u_n^\ast}{L^2(\Gamma_D)}\\
&\le C\inf_{v\in\mc{H}_n}
\Bigl((1/\gamma+1/\sqrt{\beta})
\Lr{\nm{\na(u-v)}{L^2(\Gamma_D)}+\gamma\nm{\na(u-v)}{L^2(\Om)}}+\sqrt{\beta}\nm{u-v}{L^2(\Gamma_D)}\\
&\phantom{C\inf_{v\in\mc{H}_n}}\qquad+2(1/\sqrt{\lam}+\sqrt2)\labs{I^\ast[v]-I[v]}^{1/2}\Bigr)\\
&\quad+2(1/\sqrt{\lam}+\sqrt2)\labs{I^\ast[u_n^\ast]-I[u_n^\ast]}^{1/2}.
\end{aligned}
\end{equation}
\end{theorem}

The above inequalities~\eqref{eq:error1a} and~\eqref{eq:error3} give the error estimate for the Deep Nitsche method without taking into account the iteration error in SGD. The first term in the right-hand side of both inequalities is the approximation error, which may be bounded once $\mc{H}_n$ is specified, and we shall discuss this in the next part. The second term and the third term are the estimation error, i.e., the consistency error in the sense of numerical analysis. The estimate for these terms are standard, and we refer to~\cite{Dick:2013} for a review. The difference between the above estimate and the first Strang lemma~\cite{Berger:1972} in finite element is the last term in the right-hand side of~\eqref{eq:error3}, which depends on the approximating solution $u_n^\ast$, which usual appears in the error estimate of the nonlinear problems, while there is no such term in the first lemma of Strang. 

\begin{proof}
We start with the identity~\eqref{eq:iden} and note that
\begin{align*}
\wt{I}[u-u_n^\ast]&=\wt{I}[u]+I[u_n^\ast]\\
&=\wt{I}[u]+I^\ast[u_n^\ast]+I[u_n^\ast]-I^\ast[u_n^\ast]\\
&\le\wt{I}[u]+I^\ast[v]+I[u_n^\ast]-I^\ast[u_n^\ast]\\
&=\wt{I}[u]+I[v]+I^\ast[v]-I[v]+I[u_n^\ast]-I^\ast[u_n^\ast]\\
&=\wt{I}[u-v]+\Lr{I^\ast[v]-I[v]}+\Lr{I[u_n^\ast]-I^\ast[u_n^\ast]},
\end{align*}
where we have used the minimization problem~\eqref{eq:mini} in the last step.

Proceeding along the same line that leads to~\eqref{eq:error1}, we obtain~\eqref{eq:error3}.
\end{proof}
\section{Deep Nitsche Method}\label{sec:dnm}
We minimize $I[v]$ over certain trial set $\mc{H}_n$ that will be specified below. We shall omit the subscript $n$ in what follows 
to avoid the cluttering of the notations. The resulting optimization problem is solved by the standard Stochastic Gradient Descent (SGD) method~\cite[\S 8]{Goodfellow:2016}.
\begin{equation}\label{eq:discretevara}
\Hat{u}=\argmin_{v\in\mc{H}}I[v].
\end{equation}
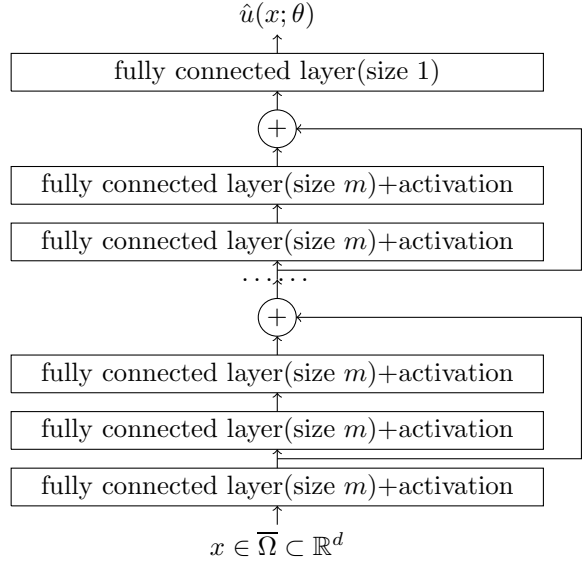
\begin{figure}[htbp]\centering\begin{tikzpicture}
	\node at(0,0){$x\in\overline{\Omega}\subset{\mb{R}^d}$};
	\draw[->](0,0.25)--(0,0.5);
	\draw(-3.5,0.5)--(3.5,0.5)--(3.5,1)--(-3.5,1)--(-3.5,0.5);
	\node at(0,0.75){fully connected layer(size $m$)+activation};
	\draw[->](0,1)--(0,1.25);
	\draw(-3.5,1.25)--(3.5,1.25)--(3.5,1.75)--(-3.5,1.75)--(-3.5,1.25);
	\node at(0,1.5){fully connected layer(size $m$)+activation};
	\draw[->](0,1.75)--(0,2);
	\draw(-3.5,2)--(3.5,2)--(3.5,2.5)--(-3.5,2.5)--(-3.5,2);
	\node at(0,2.25){fully connected layer(size $m$)+activation};
	\draw[->](0,2.5)--(0,2.75);
	\draw(0,3)circle[radius=0.25];
	\node at(0,3){+};
	\draw[->](0,1.125)--(4,1.125)--(4,3)--(0.25,3);
	\draw[->](0,3.25)--(0,3.5);
	\node at(0,3.5){……};
	\draw[->](0,3.5)--(0,3.75);
	\draw(-3.5,3.75)--(3.5,3.75)--(3.5,4.25)--(-3.5,4.25)--(-3.5,3.75);
	\node at(0,4){fully connected layer(size $m$)+activation};
	\draw[->](0,4.25)--(0,4.5);
	\draw(-3.5,4.5)--(3.5,4.5)--(3.5,5)--(-3.5,5)--(-3.5,4.5);
	\node at(0,4.75){fully connected layer(size $m$)+activation};
	\draw[->](0,5)--(0,5.25);
	\draw(0,5.5)circle[radius=0.25];
	\node at(0,5.5){+};
	\draw[->](0,3.625)--(4,3.625)--(4,5.5)--(0.25,5.5);
	\draw[->](0,5.75)--(0,6);
	\draw(-3.5,6)--(3.5,6)--(3.5,6.5)--(-3.5,6.5)--(-3.5,6);
	\node at(0,6.25){fully connected layer(size $1$)};
	\draw[->](0,6.5)--(0,6.75);
	\node at(0,7){$\Hat{u}(x;\theta)$};
\end{tikzpicture}
\caption{The component of ResNet.}\label{fig:ResNet}
\end{figure}
The trial functions set $\mc{H}$ is modeled by ResNet~\cite{He:2016}. The component of ResNet is shown in Figure~\ref{fig:ResNet}. The input layer is a fully connected layer with $m$ hidden nodes, which maps $x$ from $\mb{R}^d$ to $\mb{R}^m$. Assume that $\sigma$ is a scalar activation function and let $\phi$ be the tensor product of $\sigma$ as $\phi(x)=(\sigma(x_1),\cdots,\sigma(x_m))\in\mb{R}^m$, then
\[
s_1=\phi(W_1x+b_1),
\]
where $W_1\in\mb{R}^{m\x d}$ and $b_1\in\mb{R}^m$. The hidden layers is constructed by $l$ residual blocks. Each block contains two fully connected layers and one residual connection layer. The $i-$th block takes the form
\[
s_{i+1}=\phi(W_{2,i}\phi(W_{1,i}s_i+b_{1,i})+b_{2,i})+s_i,
\]
where $W_{1,i},W_{2,i}\in\mb{R}^{m\x m}$ and $b_{1,i},b_{2,i}\in\mb{R}^m$.

The output layer is a fully connected layer with one hidden node. The approximation solution may be represented as
\[
\Hat{u}(x;\theta)=W_2s_{l+1}+b_2,
\]
where $W_2\in\mb{R}^{1\x m}$ and $b_2\in\mb{R}$, the parameter set $\theta$ is defined as
\[
\theta{:}=\set{W_1,W_2,b_1,b_2,W_{1,i},W_{2,i},b_{1,i},b_{2,i}}{i=1,\dots,l}.
\]

In each step of the SGD iteration, we randomly sample $N_i$ points $\{x_k^{(i)}\}_{k=1}^{N_i}\subset\Omega$,  $N_d$ points $\{x_k^{(d)}\}_{k=1}^{N_d}\subset\Gamma_D$ and $N_n$ points $\{x_k^{(n)}\}_{k=1}^{N_n}\subset\Gamma_N$. The loss function is defined as
\begin{align*}
I^\ast[v]{:}=&\dfrac{\abs{\Omega}}{N_i}\sum_{k=1}^{N_i}\Lr{\dfrac12A(x_k^{(i)})
\na_x\Hat{u}(x_k^{(i)};\theta)\cdot\na_x\Hat{u}(x_k^{(i)};\theta)-f(x_k^{(i)})\Hat{u}(x_k^{(i)};\theta)}\\
&+\dfrac{\abs{\Gamma_D}}{N_d}\sum_{k=1}^{N_d}
\Lr{\dfrac{\beta}{2}\Hat{u}^2(x_k^{(d)};\theta)-\Hat{u}(x_k^{(d)};\theta)A(x_k^{(d)})
\na_x\Hat{u}(x_k^{(d)};\theta)\cdot n}\\
&-\dfrac{\abs{\Gamma_D}}{N_d}\sum_{k=1}^{N_d}g_D\Lr{\beta\Hat{u}(x_k^{(d)};\theta)
-A(x_k^{(d)})\na_x\Hat{u}(x_k^{(d)};\theta)\cdot n}\\
&-\dfrac{\abs{\Gamma_N}}{N_n}\sum_{k=1}^{N_n}g_N\Hat{u}(x_k^{(n)};\theta).
\end{align*}
Now the minimization problem reads as
\[
\min_{v\in\mc{H}}I^\ast[v].
\]

In what follows, we discuss the theoretical assumptions on Theorem~\ref{thm:main}. The inverse trace inequality~\eqref{eq:inverse} seems to has been absent for neural network functions, which is even missing for the meshless methods; c.f.,~\cite[\S 7]{Melenk:2005}. For a function $v$ representing by a two-layer Gaussian network, {\sc Mhaskar}~\cite{Mhaskar:2004, Mhaskar:2005} and {\sc Erd\'elyi}~\cite{Erde:2006} proved the following inverse inequality:
\[
\nm{\na^2 v}{L^2(\R)}\le C_{\text{inv}}\sqrt{n}\nm{\na v}{L^2(\R)},
\]
where $C_{\text{inv}}$ is an absolute constant, and $n$ is the number of the neurons. However, it does not seem easy to extend the proof of the above inequality in~\cite[Theorem 2.2]{Erde:2006} to a finite interval. Once the above type inequality is valid for a bounded domain with $\sqrt{n}$ replaced by $n^\al$ with $\al>1/2$, then we may use the trace inequality to prove ~\eqref{eq:inverse} with $\gamma=C_{\text{trace}}(1+C_{\text{inv}})n^\al$, where $C_{\text{trace}}$ is a constant appears in the classical trace inequality~\cite{Adams:2003}, which may depend on $\Omega$ but independent of $v$ and $n$.

The approximation estimates~\eqref{eq:appass} for two-layer neural networks is well-established, and we refer to~\cite{Pinkus:99} for a review. While such results for multi-layer neural networks are not so complete, we refer to~\cite{Guhring:2020, Bolcskei:2019} for the progress in this direction. Most estimates in this vein except~\cite{Lu:2020} is asymptotical in the sense that the approximation bound is valid for specified width and depth of a multi-layer neural network. But what we need is a sharp approximation bound for a deep neural network with arbitrary width and depth in the energy norm. The way for bounding the estimation error is quite standard~\cite{Dick:2013} provided that we assume certain smoothness on the functions in $\mc{H}_n$. Therefore, we may obtain the rate of convergence by combining these two type errors, and we shall leave it for further study.
\section{Numerical Experiments}\label{sec:numerics}
We apply the Deep Nitsche Method to solve the mixed boundary value problem~\eqref{eq:mixbvp}. In all the examples, we use the activation function $\sigma=\tanh$, and let the domain $\Omega$ be a hypercube unless otherwise stated, i.e., $\Omega=(0,1)^d$. We report the relative errors
\[
e_{L^2}{:}=\dfrac{\nm{u-\Hat{u}}{L^2}}{\nm{u}{L^2}}\quad\text{and\quad}
e_{H^1}{:}=\dfrac{\nm{u-\Hat{u}}{H^1}}{\nm{u}{H^1}}
\]
for all examples.

There are lots of choices for generating sampling points to approximate the energy functional $I[v]$ during the training process, which is a crucial part for the efficiency of the method. The number of the uniform sampling points grows exponentially with the dimension, hence it quickly becomes unpractical. We use Quasi-Monte Carlo method~\cite{Dick:2013} to approximate $I$, and the sampling points are generated by the Halton sequence~\cite{Halton:1960}. To be more specific, we use the growing prime values starting from $2$ as the prime bases in each dimension. For example, in three dimensional problem, the prime base is $2$ for $x$-axis, $3$ for $y$-axis, and $5$ for $z$-axis, respectively. There are many other sets of the low discrepancy sequences beyond the Halton sequence. However, an evaluation of their efficiency seems beyond the scope of the present work. Quasi-Monte Carlo is also used to approximate the relative errors $e_{H^1}$ and $e_{L^2}$ during the test process. We generate $10^5$ sampling points in $\Omega$ by the Halton sequence, with the same prime bases as above.
\subsection{Two-dimensional examples}
The solution $u$ is approximated by a neural network with five residual blocks and $10$ hidden nodes per fully connected layer. Noticing that one residual block contains $2$ fully connected layers and $1$ residual connection, the number of trainable parameters is $1141$. An Adam optimizer is employed to train with the learning rate $0.001$~\cite{Adam:2014}. We train the model for $50000$ epochs. For simplicity, we take the same number of points in the domain and on each of boundaries. In each epoch, we generate $64$ points inside the domain $\Omega$ and $64$ points on each edge of $\pa\Omega$, by a Quasi-Monte Carlo method based on low-discrepancy Halton sequence as discussed above. In order to save memory and running time, and at the same time to achieve a basically equivalent convergence effect, we set hyper-parameters, such as the number of layers and neural nodes, and the size of training set very small, because we train a relatively small neural network in this part.

In the first example, we test a mixed boundary value problem with the coefficient matrix $A$ given by
\[
A=\begin{pmatrix}
(x+1)^2+y^2 & -xy\\
-xy & (x+1)^2
\end{pmatrix}.
\]
We let the solution be
\[
u(x,y)=x^3y^2+x\sin(2\pi xy)\sin(2\pi y),
\]
and we set $\Gamma_D=\{1\}\x(0,1)\cup(0,1)\x\{1\}$ and $\Gamma_N=\{0\}\x(0,1)\cup(0,1)\x\{0\}$. The source term $f$ and the boundary data $g_D$ and $g_N$ are computed by~\eqref{eq:mixbvp}. The relative errors $e_{L^2}$ and $e_{H^1}$ decrease with the number of iterations, as shown in Figure~\ref{fig:mixed}, and the final relative errors are reported in Table~\ref{tab:mixed}, with different panalized parameters $\beta$.
\begin{figure}[htbp]\centering
\includegraphics[width=\textwidth]{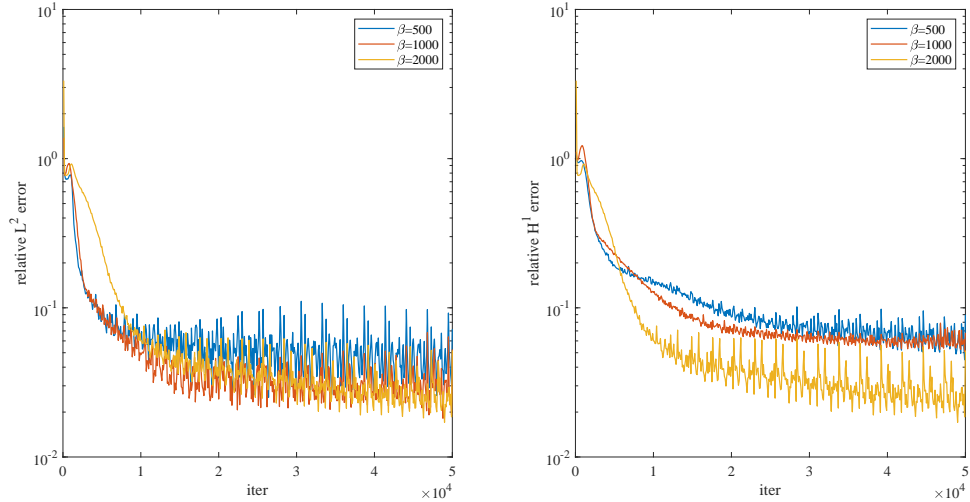}
\caption{A mixed boundary problem in two dimension with different penalized parameter $\beta$.}\label{fig:mixed}
\end{figure}

\begin{table}[htbp]\centering\caption{A mixed boundary value problem in two dimension with different penalized parameter $\beta$.}\label{tab:mixed}
\begin{tabular}{c|cc}
\hline
$\beta$ & $e_{L^2}$ & $e_{H^1}$\\
\hline
500 & 3.925e-02 & 6.960e-02\\
1000 & 3.633e-02 & 7.981e-02\\
2000 & 2.036e-02 & 5.124e-02\\
\hline
\end{tabular}\end{table}

In view of Figure~\ref{fig:mixed} and Table~\ref{tab:mixed}, it seems the accuracy for methods with different parameters $\beta$ are comparable, and method with bigger $\beta$ yields slightly better results.

In the second example, we consider
\[
\left\{
\begin{aligned}
-\triangle u(x) &=0\qquad&&x\in\Omega,\\
u(x)=u(r,\theta)&=r^{1/2}\sin(\theta/2)\qquad&&x\in\pa\Omega,
\end{aligned}\right.
\]
where $\Omega=(-1,1)^2\backslash[0,1)\times\{0\}$ is a cracked domain. This problem has an analytical solution
\[
u(x)=r^{1/2}\sin(\theta/2),
\]
which belongs to $H^s(\Omega)$ with $s<3/2$. In fact, such solution usually stands for the singular part of the general solution~\cite{strang1973analysis}. We report the relative errors in Figure~\ref{fig:lowregularity} and Table~\ref{tab:lowregularity} with different penalized parameters $\beta$.
\begin{figure}[htbp]\centering
\includegraphics[width=\textwidth]{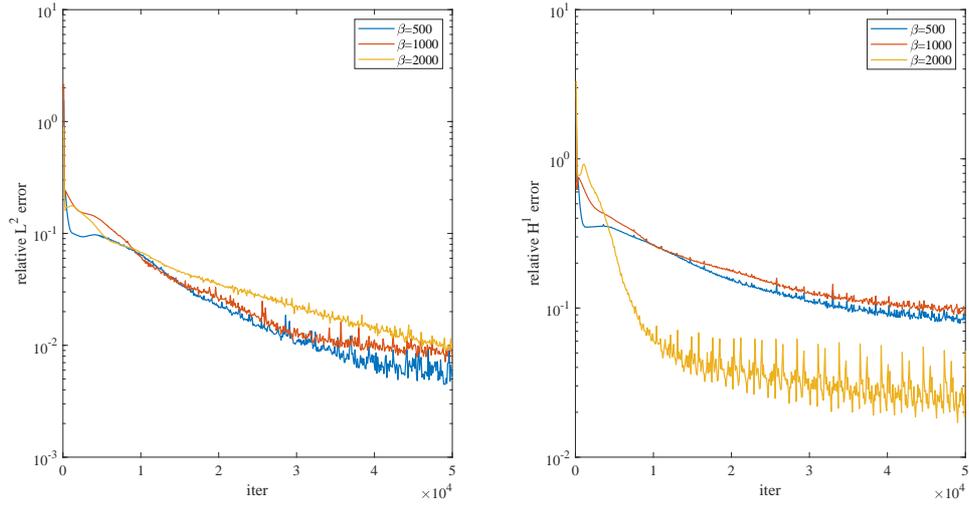}
\caption{A singular solution in two dimension.}\label{fig:lowregularity}
\end{figure}

\begin{table}[htbp]\centering\caption{A singular solution in two dimension.}\label{tab:lowregularity}
\begin{tabular}{c|cc}
\hline
$\beta$ & $e_{L^2}$ & $e_{H^1}$\\
\hline
500 & 5.298e-03 & 8.513e-02\\
1000 & 7.226e-03 & 9.148e-02\\
2000 & 1.019e-02 & 5.177e-02\\
\hline
\end{tabular}\end{table}

In view of Figure~\ref{fig:lowregularity} and Table~\ref{tab:lowregularity}, the accuracy for the methods with three different parameters $\beta$ are also comparable. By contrast to the previous example, the method with smaller parameter gives better $L^2$ error.
%
\subsection{$p$-Laplace equation}
In this part, we solve the $p$-Laplace equation posed on a unit square $\Omega=(0,1)^2$,
\[
\left\{
\begin{aligned}
-\div\Lr{\abs{\na u}^{p-2}\na u} =f\qquad&&x\in\Omega,\\
u(x)=g_D\qquad&&x\in\pa\Omega,
\end{aligned}\right.
\]
The energy functional $I$ is written as
\begin{align*}
I[v]&=\dfrac1p\int_{\Om}\abs{\na v}^p\dx+\dfrac{\beta}{2}\int_{\pa\Om}(g_D-v)^2\md\sigma(x)
+\int_{\pa\Om}(g_D-v)\pa_{\nu}v\md\sigma(x)\\
&\quad-\Lr{\int_{\Om}fv\dx+\dfrac{\beta}{2}\int_{\pa\Om}g_D^2\md\sigma(x)},
\end{align*}
where $\pa_{\nu}v=\abs{\na v}^{p-2}\pa_nv$.

In order to reduce unnecessary effects of hyper-parameters,  we keep the configurations of the neural networks the same as that for the linear problem. 

Firstly we test the classical example from~\cite{Glowinski:1975}
\[
u=2^{-1/(p-1)}(1-1/p)\Lr{1-r^{p/(p-1)}},\quad r=(x^2+y^2)^{1/2}.
\]
A direct calculation gives that $f=1$. We note that $\na u(x)\to 0$ as $x\to 0$. Such singularity may cause difficulties in computation. We report the relative errors in Figure~\ref{fig:nonlinear1} and Table~\ref{tab:nonlinear1} with different parameters $p$ and $\beta$.
\begin{figure}[htbp]
\centering
\includegraphics[width=\textwidth]{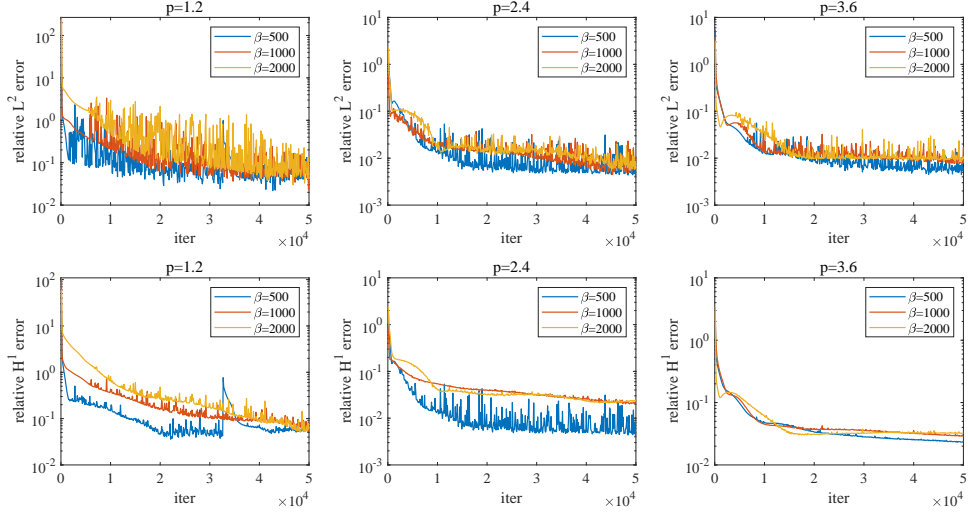}
\caption{A smooth solution for $p$-Laplace equation: for $p=1.2, 2.4$ and $3.6$, the upper figures show the relative $L^2$ error and the lower figures show the relative $H^1$ error.}\label{fig:nonlinear1}
\end{figure}
\begin{table}[htbp]
\centering\caption{A smooth solution for $p$-Laplace equation.}\label{tab:nonlinear1}\begin{tabular}{c|c|cc}
\hline
$p$&$\beta$ & $e_{L^2}$ & $e_{H^1}$\\
\hline
1.2 & 500 & 6.968e-02 & 6.426e-02\\
& 1000 & 4.993e-02 & 5.760e-02\\
& 2000 & 4.057e-02 & 5.802e-02\\
\hline
2.4 & 500 & 4.646e-03 & 4.646e-03\\
& 1000 & 1.092e-02 & 2.410e-02\\
& 2000 & 5.835e-03 & 2.371e-02\\
\hline
3.6 & 500 & 6.310e-03 & 2.397e-02\\
& 1000 & 1.747e-02 & 2.974e-02\\
& 2000 & 9.412e-03 & 3.225e-02\\
\hline
\end{tabular}\end{table}

In the second example, we test a less smooth solution:
\[
u=r^{(p-2)/(p-1)}.
\]
A direct calculation gives that $f=0$. The solution $u$ does not belong to $H^2(\Omega)$ when $p>2$, and $\na u\to\infty$ as $x\to 0$. We report the relative errors in Figure~\ref{fig:nonlinear2} and Table~\ref{tab:nonlinear2}.
\begin{figure}[htbp]\centering
\includegraphics[width=\textwidth]{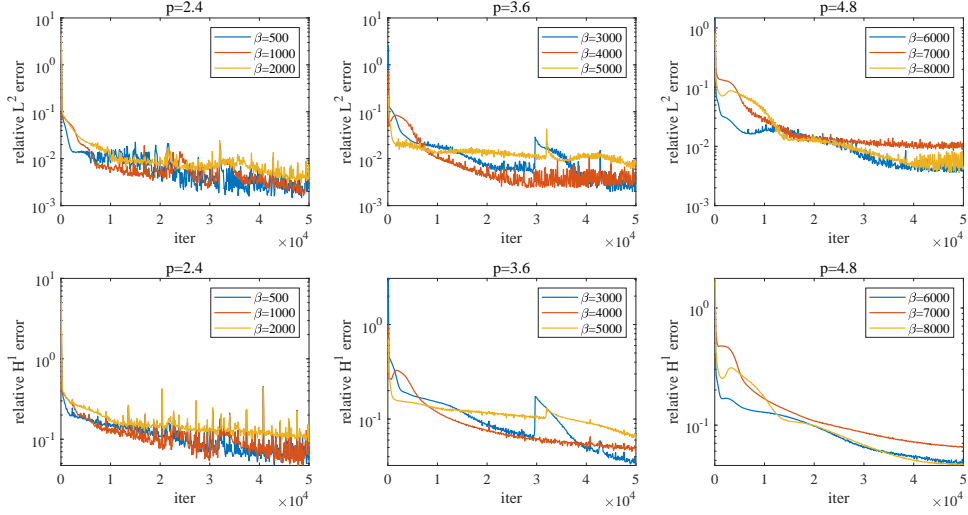}
\caption{A singular solution for $p$-Laplace equation. The upper figures show the relative $L^2$ error and the lower figures show the relative $H^1$ error.}\label{fig:nonlinear2}
\end{figure}
\begin{table}[htbp]\centering\caption{A singular solution for $p$-Laplace equation: for $p=2.4,3.6$ and $4.8$}\label{tab:nonlinear2}
\begin{tabular}{c|c|cc}
\hline
$p$&$\beta$ & $e_{L^2}$ & $e_{H^1}$\\
\hline
2.4 & 500 & 2.920e-03 & 7.758e-02\\
& 1000 & 3.786e-03 & 7.312e-02\\
& 2000 & 3.680e-03 & 1.035e-01\\
\hline
3.6 & 3000 & 3.294e-03 & 3.230e-02\\
& 4000 & 3.332e-03 & 4.699e-02\\
& 5000 & 8.085e-03 & 6.393e-02\\
\hline
4.8 & 6000 & 4.752e-03 & 4.605e-02\\
& 7000 & 9.039e-03 & 6.439e-02\\
& 8000 & 4.540e-03 & 4.577e-02\\
\hline
\end{tabular}\end{table}

In view of Figure~\ref{fig:nonlinear1}, Figure~\ref{fig:nonlinear2}, Table~\ref{tab:nonlinear1}, and Table~\ref{tab:nonlinear2}, Deep Nitsche Method equally works for the $p$-Laplace equation with small $p$ as well as large $p$. For the problem with less smooth solution, it seems wise to choose larger $\beta$ as $p$ grows.
\subsection{High-dimensional examples}
We turn to high dimensional problems. We still employ an Adam optimizer with the learning rate $0.001$ and train the model for $50000$ epochs. In each epoch, we use a Quasi-Monte Carlo method based on a low-discrepancy Halton sequence to generate $512$ points inside the domain $\Omega$ and $64$ points on each face of $\pa\Omega$.

In the first example, we consider the problem~\eqref{eq:mixbvp} on a hypercube $\Omega=(0,1)^{20}$ with pure Dirichlet boundary condition, for which we choose $f$ and $g_D$ such that the solution to~\eqref{eq:mixbvp} is given by 
\[
u(x)=\Lr{\sum_{i=1}^{20}x_i^2}^{5/2}.
\]
This example in three dimension has been test in~\cite{Griebel:2003} with a particle-partition of unit method. We approximate the solution by a neural network with five residual blocks and $50$ hidden nodes per fully connected layer. Thus the number of trainable parameters is $26601$. We report the relative errors in Figure~\ref{fig:singular20} and Table~\ref{tab:singular20} with different penalized parameters $\beta$.
\begin{figure}[htbp]\centering
\includegraphics[width=\textwidth]{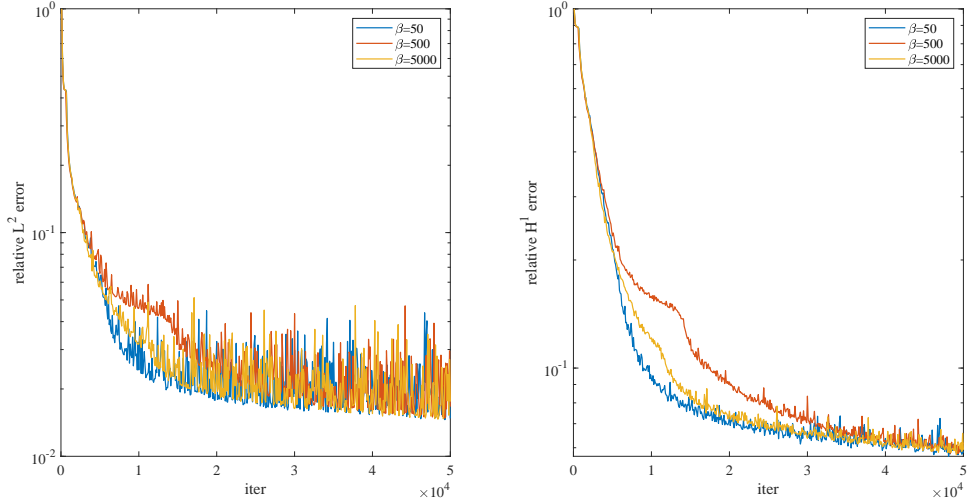}
\caption{Less smooth solution in $20$-dimension.}\label{fig:singular20}
\end{figure}

\begin{table}[htbp]\centering\caption{Less smooth solution in $20$ dimension.}\label{tab:singular20}\begin{tabular}{c|cc}
\hline
$\beta$ & $e_{L^2}$ & $e_{H^1}$\\
\hline
50 & 1.477e-02 & 5.807e-02\\
500 & 1.668e-02 & 6.137e-02\\
5000 & 1.756e-02 & 5.919e-02\\
\hline
\end{tabular}\end{table}

In the second example, we consider a smooth solution in $100$ dimension,
\[
u(x)=\exp\Lr{\dfrac{1}{100}\sum_{i=1}^{100}x_i},\qquad x\in\Omega{:}=(0,1)^{100}
\]
with a pure Dirichlet boundary condition. We compute $f$ and $g_D$ by~\eqref{eq:mixbvp}. The exact solution $u$ is approximated by a neural network with five residual blocks and $100$ hidden nodes per fully connected layer, and the number of trainable parameters is $111201$. We report the relative errors in Figure~\ref{fig:smooth100} and Table~\ref{tab:smooth100}.
\begin{figure}[htbp]\centering
    \includegraphics[width=\textwidth]{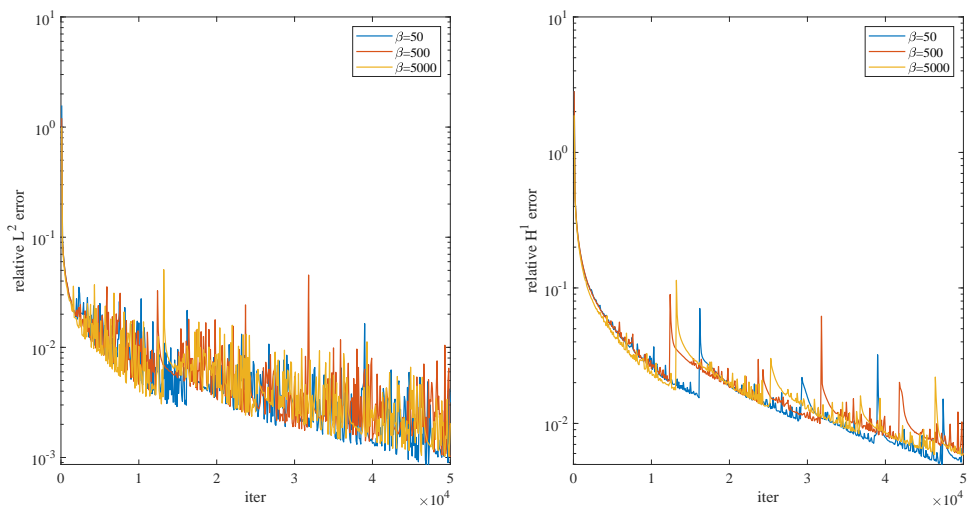}
    \caption{Smooth solution in $100$-dimension.}\label{fig:smooth100}
\end{figure}

\begin{table}[htbp]\centering\caption{Smooth solution in $100$ dimension.}\label{tab:smooth100}\begin{tabular}{c|cc}
\hline
$\beta$ & $e_{L^2}$ & $e_{H^1}$ \\
\hline
50 & 3.172e-03 & 5.924e-03\\
500 & 1.011e-03 & 5.850e-03\\
5000 & 2.049e-03 & 6.028e-03\\
\hline
\end{tabular}\end{table}

Figure~\ref{fig:singular20}, ~\ref{fig:smooth100} and Table~\ref{tab:singular20}, ~\ref{tab:smooth100} show that our method has potential to work for boundary value problems in rather high dimension.
\section{Conclusion}
Based on Nitsche's idea and representing the trial functions by deep neural networks, we propose a new method to deal with the complicated boundary conditions for boundary value problems. The test examples and the error estimate show that the method has the following advantages:
\begin{enumerate}
\item It deals with the mixed boundary conditions in a unified variational way without significant extra costs, and it fits well with the stochastic gradient descent method. It lends itself to a rigorous error estimate. 

\item It works for the problems in low dimension as well as high dimension. It also has potential to work for problems in rather high dimension. It equally works for nonlinear problems.

\end{enumerate}

Besides the above remarks, it would be an interesting direction to extend the method to deal with the time-dependent problems, in particular for problems with mixed time varying boundary conditions. It is of great interest to prove the convergence rate of the method though we have derived an energy error bound. These will be left as future work.

\end{document}